\documentclass{amsart}

\usepackage[utf8]{inputenc}
\usepackage{amsmath}
\usepackage{amsfonts}
\usepackage{amsthm}
\usepackage{authblk}
\usepackage{hyperref}
\usepackage{resizegather}
\usepackage{graphicx} 
\usepackage{enumerate}
\usepackage{amssymb}
\usepackage{amsaddr}
\usepackage{mathtools}
\usepackage{arydshln}
\usepackage{algorithm}
\usepackage{algpseudocode}
\usepackage{float}
\usepackage{xcolor}
\usepackage{geometry}
\usepackage{subfig}
\usepackage{wrapfig}
\usepackage{multirow}
\usepackage{float}

\newtheorem{thm}{Theorem}[section]
\newtheorem{cor}[thm]{Corollary}

\newtheorem{lem}[thm]{Lemma}

\newtheorem{prob}[thm]{Problem}
\newtheorem{ques}[thm]{Question}

\newtheorem{defn}[thm]{Definition}
\newtheorem{examp}[thm]{Example}
\newtheorem{notn}[thm]{Notation}
\newtheorem{obs}[thm]{Observation}
\theoremstyle{remark}
\newtheorem{rem}[thm]{Remark}

\numberwithin{equation}{section}

\algrenewcommand\algorithmicrequire{\textbf{Input:}}
\algrenewcommand\algorithmicensure{\textbf{Output:}}
\algnewcommand{\Downto}{\textbf{ downto }}

\DeclareMathOperator*{\argmax}{argmax}

\numberwithin{equation}{section}

\newcommand{\PP}{\mathcal P}
\newcommand{\snk}{s^{n,k}}
\newcommand{\ap}{AP_{n,k}}
\newcommand{\A}{\mathcal{A}}

\newcommand{\slack}{\ensuremath{\textrm{slack}}}

\title{An infinite family of counterexamples to a conjecture on distance magic labeling}

\begin{document}

\maketitle

\begin{center}
  {\Large Ehab Ebrahem, Shlomo Hoory and  Dani Kotlar}\\
  {\small nbehabibraheem@gmail.com, hooryshl@telhai.ac.il and dannykot@telhai.ac.il}\\
  Department of Computer Science, Tel-Hai College, Israel
  \par \bigskip
\end{center}

\begin{abstract}
This work is about a partition problem which is an instance of the distance magic graph labeling problem.
Given positive integers $n,k$ and $p_1\le p_2\le \cdots\le p_k$ such that $p_1+\cdots+p_k=n$ and $k$ divides $\sum_{i=1}^ni$, we study the problem of characterizing the cases where it is possible to find a partition of the set $\{1,2,\ldots,n\}$ into $k$ subsets of respective sizes $p_1,\dots,p_k$, such that the element sum in each subset is equal. 
Using a computerized search we found examples showing that the necessary condition,  $\sum_{i=1}^{p_1+\cdots+p_j} (n-i+1)\ge j{\binom{n+1}{2}}/k$ for all $j=1,\ldots,k$, 
is not generally sufficient, refuting a past conjecture. 
Moreover, we show that there are infinitely many such counter-examples.
The question whether there is a simple characterization is left open and for all we know the corresponding decision problem might be NP-complete.
\end{abstract}

\noindent {\bf Keywords:} Distance magic labeling, equitable partition and weighted fair queuing (WFQ).
\section{Introduction}
Let $G=(V,E)$ be a finite, simple, undirected graph of order $n$. Denoting $[n]=\{1,2,\ldots,n\}$, a \emph{distance magic labeling} of $G$ \cite{miller2003} (or \emph{$\Sigma$-labeling} \cite{beena2009}) is a bijection $f:V\rightarrow[n]$ such that for all $x\in V$, $\sum_{y\in N(x)}f(y)=\mathbf{c}$ for some constant $\mathbf{c}$. 
The notion of distance magic labeling has been extensively studied for many families of graphs (see \cite{gallian2018dynamic}, \cite{anholcer2015distance}, \cite{arumugam2012distance} for surveys). The following problem is equivalent to determining whether a complete multipartite graph \cite{miller2003}, and other types of graphs \cite{anholcer2016spectra}, has a distance magic labeling:

\begin{prob}[\cite{anholcer2016spectra}, Problem 6.3]\label{prob1}
Let $n,k$ and $p_1\le p_2\le \cdots\le p_k$ be positive integers such that $p_1+\cdots+p_k=n$ and $k$ divides $\sum_{i=1}^ni$. When is it possible to find a partition of the set $[n]$ into $k$ subsets of respective sizes $p_1,\dots,p_k$, such that the element sum in each subset is equal?
\end{prob}

Anholcer, Cichacz and Peterin \cite{anholcer2016spectra} observed that a necessary condition for such a partition to exist is that:
\begin{equation}\label{eq:nec_cond}
    \sum_{i=1}^{p_1+\cdots+p_j} (n-i+1)- j\frac{\binom{n+1}{2}}{k}\ge 0\quad\textrm{for all}\quad j=1,\ldots,k.
\end{equation}

It is known that this condition is sufficient whenever $k=2,3$ (\cite{beena2009} and \cite{miller2003}) and $k=4$ \cite{kotlar16}. Consequently, it was conjectured in 
\cite{miller2003} and \cite{kotlar16} that this holds in general. 
In this work, we used computing to refute this conjecture and to find infinite families of counter-examples. 

The outline of the paper is as follows: In Section~\ref{sec:definitions} we introduce definitions and notation; 
in Section~\ref{sec:constructing} we describe the algorithms we used to find a partition of $[n]$ as described in  problem~\ref{prob1}, given $n,k$ and $p_1\le p_2\le \cdots\le p_k$; 
in Section~\ref{sec:criteria} we introduce two criteria for identifying cases where such a partition does not exist (and thus, refuting the above-mentioned conjecture); in Section~\ref{sec:implementation} we describe the computational results of running the algorithms and implementing the criteria on all relevant partitions for $n\le 200$; In Section~\ref{sec:special_families} we describe some infinite families of counter-examples; In Section~\ref{sec:conclusion} we point at possible directions for future research and introduce a few open questions.


\section{Definitions and Notation}\label{sec:definitions}

\begin{defn}\label{def:ap}
Let $n$ and $k$ be positive integers. An ascending partition of $n$ of size $k$ is a sequence of positive integers $\PP=[p_1,p_2,\ldots, p_k]$ where $p_1\le p_2\le\cdots\le p_k$ and $\sum_{i=1}^k p_i = n$.
\end{defn}


\begin{defn}\label{def:eq}
Let $\PP$ be a size $k$ ascending partition of $n$. We say that $\PP$ is equitable if it is possible to partition $[n]$ into subsets of respective sizes $p_1,p_2,\ldots, p_k$ so that the sum of elements in each part is the same. In this case, the resulting partition of $[n]$ is said to be balanced.
\end{defn}

\begin{notn}
    For a size $k$ ascending partition $\PP$ of $n$ to be equitable, the number of parts $k$ must divide $\sum_{i=1}^n i$. In this case, we denote by $\snk$ the sum of each part of the balanced partition of $[n]$. That is, 
\[
s^{n,k}:=\frac{n(n+1)}{2k}
\]
\end{notn}
 
\begin{notn}
    If $n$ and $k$ are such that $\snk$ is an integer, we denote by $AP_{n,k}$ the family of all ascending partitions of $n$ of size $k$. Otherwise, $AP_{n,k}$ is the empty set. 
\end{notn}

\begin{defn}\label{def:slack}
    Let $\PP=[p_1,p_2,\ldots, p_k]\in AP_{n,k}$. For $j=1,\ldots,k$ we denote 
    $$\textrm{slack}_j(\PP) := \sum_{i=1}^{p_1+\cdots+p_j} (n-i+1)- js^{n,k}.$$ 
    This expression is `the slack of $\PP$ at $j$'. We define 
    $$\textrm{slack}(\PP):=\min_{1\le j \le k-1}\textrm{slack}_j(\PP)$$
    and call this expression `the slack of $\PP$'.

\end{defn}

Note that always $\textrm{slack}_k(\PP)=0$, so this term is excluded from the minimum in the definition of $\textrm{slack}(\PP)$.

In terms of this new notation the necessary condition in \eqref{eq:nec_cond} can be formulated as follows:

\begin{obs}
 Let $\PP$ be a size $k$ ascending partition of $n$. If $\PP$ is equitable, then ${\textrm slack}(\PP) \ge 0$.
\end{obs}

\begin{defn}
    We call the condition ${\textrm slack}(\PP)\ge0$ the `slack condition'. If this condition holds for $\PP\in AP_{n,k}$ we say that $\PP$ satisfies the slack condition, or that $\PP$ is SSC.
\end{defn}

Problem~\ref{prob1} asks: when is $\PP\in AP_{n,k}$ equitable? As mentioned above, it was shown that the slack condition is sufficient for $\PP$ in $AP_{n,2}$, in $AP_{n,3}$  (\cite{miller2003} and \cite{beena2009}) and in $AP_{n,4}$ \cite{kotlar16} to be equitable, for all $n$. However, we shall show that this is not the case in general, and there are infinitely many counterexamples.

For brevity, we adopt the following notation:

\begin{defn}
    Let $\PP$ be an ascending partition. 
    The notation $\PP=[q_1^{e_1},q_2^{e_2},\ldots, q_t^{e_t}]$ for $q_1<q_2<\cdots<q_t$ means that for $i=1,\ldots,t$, $\PP$ has $e_i$ parts of size $q_i$. Each $q_i^{e_i}$ is called a block, or the $q_i$-block. 
\end{defn}

\begin{examp}\label{ex:39:13}
The partitions $[2^9, 3^2, 5^1, 10^1]$, $[2^9, 3^2, 6^1, 9^1]$, $[2^9, 3^2, 7^1, 8^1]\in AP_{39,13}$ are SSC but non-equitable.
\end{examp}

\begin{proof}
It can be verified that ${\textrm slack}(\PP)\ge0$ in each of the three cases and that $s^{39,13}=60$. Suppose there is a partition of $[39]$ containing nine parts of size 2 and two parts of size 3, so that all its parts' sums are 60. In this case, the parts of size 2 must contain all numbers $21,22,\ldots,39$, excluding 30, in pairs of sum 60. On the other hand any part of size $3$ must contain at least one number greater than $20$. Since there are two parts of size 3 and only one number left that is greater than 20, such a partition is impossible. 
\end{proof}

\noindent
The partitions in Example~\ref{ex:39:13} are the non-equitable SSC partitions with the smallest value of $n$.

\begin{defn}\cite{kotlar16}
    Let $\PP=\{p_1, p_2,\cdots, p_k\}$ be a size $k$ ascending partition of $n$. 
    We say that the partition $\A=\{A_1,A_2,\dots,A_k\}$ of $[n]$ \emph{implements}  $\PP$, if  $|A_i|=p_i$ for all $i=1,\ldots,k$.
\end{defn}

\begin{notn}\cite{kotlar16}
    For any set of numbers $A$, we denoted by $S(A)$ the sum of the elements in $A$. 
    For a partition $\A=\{A_1,A_2,\dots,A_k\}$ of $[n]$, we define $d(\A):=\sum_{i=1}^k (S(A_i)-s^{n,k})^2$.
\end{notn}
\noindent So, $\A$ is balanced if and only if $d(\A)=0$.

\section {Constructing balanced partitions}\label{sec:constructing}

In this section, we describe our algorithms for searching a balanced partition of $[n]$ implementing a given SSC partition in $AP_{n,k}$. 
This search enabled us to find the partitions in Example~\ref{ex:39:13} and many others. 

The problem is to find a practical computational procedure for checking if a given SSC partition $\PP\in AP_{n,k}$ is equitable. 
Since a brute force search over all possible partitions of $[n]$ implementing $\PP$ is not practical, we seek some heuristics for constructing such a balanced partition, if it exists.
The algorithms we consider have two main stages:

\begin{enumerate}
    \item Generating an initial partition $\A$ implementing $\PP$.\label{stage1}
    \item Reducing the value of $d(\A)$ by repeatedly exchanging elements among the parts of $\A$, aiming to eventually attain $d(\A)=0$.\label{stage2}
\end{enumerate}

\subsection{Initialization methods}\label{sec:init}

For stage (1) we can either use a general-purpose initialization method (methods (a)-(c) below) or an initialization method specifically designed to obtain a low value of $d(\A)$ (method (d)). 

\begin{enumerate}
    \item [(a)] {\bf Random placement:}
        The sets $A_1,A_2,\dots,A_k$ are a uniform random partition of $[n]$ with $|A_i| = p_i$.
    {\bf     }
    \item [(b)] {\bf Whole sets placement:}

        The numbers $n,n-1,\ldots,2,1$ are placed in the sets $A_1,A_2,\dots,A_k$, in this order, so that numbers are not placed in $A_j$ before $A_1,\ldots,A_{j-1}$ are filled to their allotted sizes $p_1,\ldots,p_{j-1}$. For example, for the partition $\PP=\{2,3,4\}$ the initial partition of $[9]$ will be $\{\{9,8\},\{7,6,5\},\{4,3,2,1\}\}$.

    \item [(c)] {\bf Round-robin placement:}

    The numbers $n,n-1,\ldots,2,1$ are placed in round-robin order, starting from $A_1$ down to $A_k$ and then back to $A_1$ and so forth, placing elements in sets that are not yet full. For example, for the partition $\PP=\{2,3,4\}$ the initial partition of $[9]$ will be $\{\{9,6\},\{8,5,3\},\{7,4,2,1\}\}$.

    \item [(d)] {\bf Adaptive Round Robin placement algorithm:}

An approach based on the well-known Weighted Fair Queuing (WFQ) scheduling algorithm~(\cite{tanenbaum2011computer} chapter 5). 
Starting from $n$ and going down to 1, we assign each number to the set that is most `in need', where the measure for 'need' is the average number needed to complete the sum of the elements in the set to $\snk$. More precisely, for a set $A_j$ this is the amount needed to reach $\snk$ divided by the number of slots still available in $A_j$. If at some stage $A_j$ has already been assigned $n_j<p_j$ elements, whose sum is $s_j$, then the 'need' of $A_j$ is 
\begin{equation}\label{def:need:original}
    n(A_j) = (\snk-s_j)/(p_j-n_j). 
\end{equation}
Thus, the set that gets the next number is the one with the maximum need.

Note that when referring to a set $A_j$ we are using the term `set' loosely, not as a collection of some fixed elements, but rather as a collection of a fixed number of slots in which we can place elements (an element in each slot). Initially, these slots are empty and we fill them, one-by-one in the process. 

As done in the WFQ algorithm, the procedure described above can be improved upon. 
So, instead of calculating the current 'need' we calculate the resulting need if the next number is assigned to the set. So, if the next number to be placed is $x$ and $A_j$, $n_j$ and $s_j$ are as above, then we revise \eqref{def:need:original} to:

\begin{equation}\label{def:need:fixed}
    n(A_j,x) = (\snk-s_j-x)/(p_j-n_j-1), 
\end{equation}
where in the case that $p_j=n_j+1$ we set the denominator of \eqref{def:need:fixed} to 1, to avoid division by zero.

\begin{algorithm}[H]
\caption{The Adaptive Round Robin algorithm for placing the numbers $1,\ldots,n$ in sets of given sizes so that the sums of the sets are nearly equal}\label{alg:adaptiveRoundRobin}
\begin{algorithmic}[1]
\Require $n, k$
\Require $p[1], \ldots,p[k]$ \Comment{A non-descending list of positive integers that sum up to $n$}
\Ensure $partition$ \Comment{A list of $k$ lists where the size of partition[$j$] is p[$j$] for all $j$}
\State $s \gets \textrm{sum}([1,2,\ldots,n])/k$
\State $indexes \gets [1,2,\ldots,k]$
\State $partition \gets [\;[\;],\ldots,[\;]\;]$ \Comment{A list of $k$ empty lists, indexed from 1 to $k$}
\For{$i \gets n \Downto 1$}
    \If {for some $t\in indexes$, $\textrm{sum}(partition[t])=s-i$ and $\textrm{size}(partitions[t])=p[t]-1$}
        \State $j \gets t$ \Comment {the $t$th part will be completed}
    \Else
        \State $j \gets \argmax((s-\textrm{sum}(partition[t])-i)/\max((p[t]-len(partition[t])-1),1))$ for $t\in indexes$
    \EndIf
    \State $\textrm{AddItem}(partition[j],i)$
    \If {$\textrm{size}(partition[j]) = p[j]$}
    \State $\textrm{RemoveItem}(indexes,j)$
    \EndIf

\EndFor
\end{algorithmic}
\end{algorithm}

\end{enumerate}

\subsection{Reducing the value of $d(\A)$ by transpositions}\label{sec:reducing}

After initializing a partition $\A$ of $[n]$ we start making modifications with the goal of obtaining a balanced partition. 

\begin{defn} \cite{kotlar16} 
Let $\A=\{A_1,A_2,\dots,A_k\}$ be a partition of $[n]$ and let $a,b\in[n]$. Suppose $a\in A_i$ and $b\in A_j$. 
We denote by $\chi_{a,b}$ the operator that acts on $\A$ by exchanging $a$ and $b$ between $A_i$ and $A_j$. So, if $i\ne j$, the result is a new partition $\chi_{a,b}(\A)$, where $A_i$ is replaced by $A_i\setminus\{a\}\cup \{b\}$ and $A_j$ is replaced by $A_j\setminus\{b\}\cup \{a\}$. Otherwise, $\chi_{a,b}$ is the identity. We call $\chi_{a,b}$ an exchange. If $|b-a|=1$ we call $\chi_{a,b}$ a transposition.
\end{defn}

\begin{obs} \cite{kotlar16}\label{obs:1}
Let $\A=\{A_1,A_2,\dots,A_k\}$ be a partition of $[n]$  and let $a\in A_i$, $b\in A_j$ for $i\ne j$.  Let $t=b-a$ and $u=S(A_j)-S(A_i)$. Then, 
\begin{enumerate}
    \item [(i)] $d(\chi_{a,b}(\A))=d(\A)-2t(u-t)$
    \item [(ii)]
\[d(\chi_{a,b}(\A))\begin{cases}
			=d(\A), & \text{if $b-a=S(A_j)-S(A_i)$}\\
            >d(\A), & \text{if $b-a > S(A_j)-S(A_i)$}\\
            <d(\A), & \text{if $b-a < S(A_j)-S(A_i)$}
		 \end{cases}
\]
\end{enumerate}
\end{obs}

\begin{defn}
    Let $\A$ be a partition of $[n]$ and $a,b\in[n]$. We say that $\chi_{a,b}$ is an improving exchange if $d(\chi_{a,b}(\A))<d(\A)$ and a neutral exchange if $d(\chi_{a,b}(\A))=d(\A)$.
\end{defn}

Given the initial partition $\A$, as described in the previous subsection, we aim to reduce the value of $d(\A)$, possibly down to zero, by performing improving exchanges. So, we need to find pairs $a,b$ so that $\chi_{a,b}$ is improving, that is, $b-a < S(A_j)-S(A_i)$, by Observation \ref{obs:1}. 
For convenience, we limit ourselves to transpositions. This approach is supported by the fact that any exchange is the composition of transpositions, as it can be shown, that similarly to the case of symmetry groups, for any $a<b\in [n]$,
\[\chi_{a,b} = \chi_{a,a+1}\chi_{a+1,a+2}\cdots\chi_{b-2,b-1}\chi_{b-1,b}\chi_{b-2,b-1}\cdots\chi_{a+1,a+2}\chi_{a,a+1}.\]
Note that this does not mean that any improving exchange is the composition of improving transpositions, so, we cannot guarantee that any possible improving exchange can be achieved by a series of improving transpositions, but in practice, this approach seems to work.

We shall describe a heuristic that applies transpositions in order to find a balanced partition implementing a given equitable ascending partition. The heuristic contains three types of steps:

\begin{enumerate}
    \item Improving step: perform improving transpositions as long as possible.
    \item Neutral step: If after step (1) still $d(\A)>0$ we perform a neutral transposition and go back to step (1)
    \item Perturbation step: If $d(A)>0$  and there are no improving or new neutral transpositions to be made, or after performing a fixed number of rounds of steps (1) and (2) still  $d(A)>0$, we perturb the system by performing a fixed number of random exchanges and go back to step (1).
\end{enumerate}

The algorithm stops once we obtain $d(A)=0$ or after a fixed number of rounds of step (3). In the latter case, the equitability status of the partition is unresolved.

To illustrate steps (1) and (2), consider the ascending partition $[2^3, 3^1, 4^1, 7^1]\in AP_{20,6}$. After initialization and performing some improving transpositions, we obtained the partition 
$$\{14, 20\}, \{15, 19\}, \{17, 18\}, \{9, 10, 16\}, \{5, 7, 11, 13\}, \{1, 2, 3, 4, 6, 8, 12\}$$
of $[20]$. The sum in the first two parts is 34 (1 less than $s^{20,6}=35$) and the sum in the last two parts is 36, so the partition is not balanced. The reader can verify that there are no improving transpositions to be made. If we perform the neutral transposition $\chi_{15,16}$ we obtain the partition 
$$\{14, 20\}, \{16, 19\}, \{17, 18\}, \{9, 10, 15\}, \{5, 7, 11, 13\}, \{1, 2, 3, 4, 6, 8, 12\},$$
and now we can perform the improving transposition $\chi_{10,11}$ and obtain:
$$\{14, 20\}, \{16, 19\}, \{17, 18\}, \{9, 11, 15\}, \{5, 7, 10, 13\}, \{1, 2, 3, 4, 6, 8, 12\},$$
where only the first and the last parts have a sum different from 35. 
Again we don't have an improving transposition available. So, we perform the neutral transposition $\chi_{14,15}$ to obtain the partition
$$\{15, 20\}, \{16, 19\}, \{17, 18\}, \{9, 11, 14\}, \{5, 7, 10, 13\}, \{1, 2, 3, 4, 6, 8, 12\},$$
and then the improving transposition $\chi_{11,12}$ will produce the balanced partition
$$\{15, 20\}, \{16, 19\}, \{17, 18\}, \{9, 12, 14\}, \{5, 7, 10, 13\}, \{1, 2, 3, 4, 6, 8, 11\}.$$


\hfill

The following algorithm takes a partition $\A$ of $[n]$ and attempts to reduce the value of $d(\A)$ by transpositions, aiming to obtain a balanced partition.

\begin{algorithm}[H]
\caption{Balancing a given partition of $[n]$}\label{alg:balance_partition}\label{alg:transpositions}
\begin{algorithmic}[1]
\Require $n,k$ 
\Require $Partition$ \Comment{a partition of $[n]$}
\Require $MaxStepsUntilPerturbation$ 
\Require $MaxPerturbations$ 
\Ensure $IsEquitable$ \Comment {True means that the Partition can be balanced}
\State $NeutralSteps = 0$
\State $PerturbationSteps = 0$
\State $IsEquitable = IsBalanced(Partition)$
\While{not $IsEquitable$ and $PerturbationSteps<MaxStepsBetweenPerturbations$}
    \If {$MakeImprovingTransposition(Partition) = TRUE$}
        \State $NeutralSteps = 0$
        \State $IsEquitable = IsBalanced(Partition)$
    \ElsIf {$NeutralSteps\le MaxStepsUntilPerturbation$}
        \State $MakeNeutralTransposition(Partition)$
        \State $NeutralSteps = NeutralSteps +1$
    \Else
        \State $MakePerturbation(Partition)$
        \State $PerturbationSteps=PerturbationSteps+1$
        \State $NeutralSteps=0$
    \EndIf
\EndWhile
\end{algorithmic}
\end{algorithm}

The procedure $IsBalanced$ returns true if the partition is balanced, and False otherwise. The procedure \textit{MakeImprovingTransposition} chooses an improving transposition randomly and performs it, if such a transposition exists, in which case it returns True. Otherwise, it returns False.
The procedure \textit{MakeNeutralTransposition} chooses a neutral transposition randomly, if one exists, and performs it. Finally, the procedure \textit{MakePerturbation} randomly picks two distinct sets in the partition and exchanges two random elements between them. 
This is done $m$ times. We tried several values for $m$ and found out that $m=\lfloor\sqrt{k}\rfloor$ works well.

\section{Criteria for non-equitability} \label{sec:criteria}
Verifying whether an ascending partition $\PP$ is equitable by trying to construct a balanced partition implementing $\PP$ is very time-consuming and is not guaranteed to produce the correct answer if $\PP$ is not equitable. 
As will be detailed in Section~\ref{sec:implementation}, running the software for all SSC ascending partitions for $n\le 200$, left some ascending partitions undecided, namely, it failed to find a corresponding balanced partition of $[n]$. Individually verifying that each of these partitions is non-equitable, as illustrated in Examples~\ref{ex:39:13}, is not practical, so instead we generalized that idea into a low-complexity criterion. 
In this section, we introduce two such criteria that suffice to detect all non-equitable SSC ascending partitions for $n\le200$, where each one of them is not sufficient on its own. There is no guarantee that these two criteria suffice for partitions with $n>200$.

The two criteria for identifying non-equitable partitions are restricted to $P = [2^e, p^f, \ldots] \in AP_{n,k}$ with $e,f>0$ and $p \ge 3$,
for which we introduce the following notation:

\begin{notn}\label{notn:h}
    Given the partition $P = [2^e, p^f, \ldots] \in AP_{n,k}$ with $e,f>0$ and $p \ge 3$, 
    denote the set of numbers that can be assigned to a 2-tuple in a balanced partition by $C = \{x\in [n]:x \ge c\}$ where $c = \snk-n$,
    and denote the number of $C$-elements not assigned to a 2-tuple by 
\begin{equation}\label{eq:h}
    h = |C| - 2e = 2n - \snk + 1 - 2e.
\end{equation}

\end{notn}

The idea of Criterion 1 for proving the non-equitability of the ascending partition $[p_1,p_2,\ldots,p_k]$ is that any attempt to find a balanced partition prematurely runs out of elements of the set $C$ and it is impossible to satisfy the remaining parts because of what we loosely call the ``slack without $C$ condition". Example~\ref{ex:39:13} is the first example of a non-equitable SSC partition, and the proof given there is an example of such argument.

\begin{thm}[Criterion 1]\label{thm:criterion1}
    Let $P = [2^e, p^f, \ldots] \in AP_{n,k}$ with $e,f>0$ and $p \ge 3$, with $C,c,h$ as defined in Notation~\ref{notn:h}, such that $f>h$, and
    \begin{eqnarray}
        \sum_{i=c-p(f-h)}^{c-1}i < (f-h)\snk.\label{eq:criterion1}
    \end{eqnarray}
    Then $P$ is non-equitable.
\end{thm}

\begin{proof}    
    Assume for contradiction that there exists a balanced partition $\A$ implementing $\PP$. 
    Then, all 2-tuples are assigned with $C$-elements, so there are $h$ $C$-elements left after populating the 2-tuples. 
    Therefore, there are at least $f-h > 0$ $p$-tuples not containing any $C$-element. 
    The sum of elements in these tuples is at most as large as the sum of the $(f-h)p$ largest numbers not in $C$, which by \eqref{eq:criterion1} is less than $(f-h)\snk$, 
    contrary to our assumption that $\A$ is balanced a partition.
\end{proof}

As observed in this study, there are non-equitable SSC partitions that are not explained by Criterion 1. The first such partition is the following example.  Criterion 2 generalizes the principle and is able to catch all the remaining non-equitable SSC partitions up to $n=200$.

\begin{examp}
The partition $\PP=[2^{25}, 3^2, 4, 6, 14]\in AP_{80,30}$ is SSC but non-equitable.
\end{examp}

\begin{proof}
    It can be verified that ${\textrm slack}(\PP)\ge0$. 
    We have $s^{80,30}=108$, and we know that each of the 25 pairs must be assigned two elements from the set $C=\{c,\ldots,80\}$, where $c=28$ and $|C|=53$. 
    So, assume for contradiction that there exists a balanced partition $\A$ implementing $\PP$. 
    Then, the number of $C$-elements that are left after populating the $2^{25}$ block is $h=3$. 
    So, non-equitability cannot be detected by Criterion 1, since the condition $f > h$ in Theorem~\ref{thm:criterion1} is not satisfied.
    However, the middle element of $C$ is $m=54$ and as each of the $25$ pairs gets one element larger than $m$ and one smaller, we conclude that the remaining 3 elements after the 2-block must be $x,y$ and $m$ where $x < m < y$. 
    Since $m + (c-1) + (c-2) = 107 < s^{80,30}=108$, we conclude that the block $3^2$ consumes the remaining 3 elements from $C$ and that the subsequent $4$-tuple is satisfied with elements smaller than $c$.
    However, $\sum_{i=c-4}^{c-1} i = 102 < s^{80,30}$, contradicting the assumption that $\A$ is balanced.
\end{proof}

\begin{thm}[Criterion 2]\label{thm:criterion2}
    Given $0<k<n$, such that $\snk$ is even, let $\PP = [p_1,\ldots,p_k] = [2^e, p^f,\ldots]\in \ap$ for some $e,f>0$ and $p\ge 3$. Let $C,c$, and $h$ be as defined in Notation~\ref{notn:h}. Let $m=\snk/2$, $\overline{h} = \lceil h/2\rceil$ and $\underline{h}=\lfloor h/2 \rfloor$.

    Suppose
    \begin{equation}\label{eq:low_sum_condition}
        m+\sum_{i=c-p+1}^{c-1}i<\snk
    \end{equation}
    
    and define
      \[ g = \begin{cases}
      \mbox{$h - f$}   & \mbox{ if $ f\leq\underline{h} $}\\
      \mbox{$\overline{h} - 2 (f-\underline{h})$} & \mbox{ if $ f>\underline{h} $}
      \end{cases}
      \]
      If (i) $g<0$ or (ii) $g\ge 0$ and for $q=p_{e+f+g+1}$, $\sum_{i=c-q}^{c-1}i<\snk$, then $\PP$ is non-equitable.
    
\end{thm}

\begin{proof}
    Let $\PP\in AP_{n,k}$ satisfy the given conditions and either (i) or (ii). 
    Assume for contradiction that $\PP$ is implemented by the balanced partition $\A$. 
    Let $C'\subset C$ be the set of elements of $C$ that are not assigned to the block $2^e$.
    So, $h=|C'|$ and, since $\snk$ is even, $|C|$ and $h$ are odd. So, $C'$ is nonempty and contains the number $m=\snk/2$. Let $C'_{\leq}=\{x\in C': x\le m\}$ and $C'_{>}=C'\setminus C'_{\leq}$. The set $C'$ consists, in addition to the number $m$, of pairs of numbers such that the sum of each pair is $\snk$.  
    So, $\underline{h}=|C'_{>}|$ and $\overline{h} = |C'_{\leq}| = \underline{h} + 1$.

    Condition \eqref{eq:low_sum_condition} means that if a $p$-tuple in $\A$ contains only one element of $C'$ this element must be greater than $m$, that is, it must be from $C'_{>}$. Otherwise, the $p$-tuple must contain at least two elements of $C'$. 
    The parameter $g$ is an upper bound on the number of $C'$-elements not assigned to the $p^f$ block, where $g<0$ means that there is a $p$-tuple not using any $C'$-element.
    To derive the formula for $g$, observe that the most thrifty use of $C'$-elements is to start by using $C'_{>}$ with one element per $p$-tuple and after it runs out, to switch to using $C'_{\leq}$ with two elements per $p$-tuple. If $ f\leq\underline{h} $, the procedure ends during first stage and we get $g = |C'|-f = h-f$.
    Otherwise, $g = |C| - |C'_{>}| - 2(f-|C'_{>}|) = \overline{h} - 2 (f-\underline{h})$.
    
    If condition (i) holds, it means that there is a $p$-tuple not assigned any $C$-element by $\A$. The element sum of this $p$-tuple is at most $\sum_{i=c-p}^{c-1}i < m+\sum_{i=c-p+1}^{c-1}i<\snk$, by \eqref{eq:low_sum_condition}, contradicting the assumption that $\A$ is a balanced partition. 
    If condition (ii) holds, then the first tuple not assigned any $C$-element by $\A$ is at index $i \leq e+f+g+1$ and has size at most $q$. 
    However, the extra condition in (ii) means that even populating this $q$-tuple with the largest elements not in $C$ yields a smaller sum than required. 
    Thus, again, $\A$ cannot be a balanced partition, and a contradiction follows.
\end{proof}


\section{Experimental results}\label{sec:implementation}

Our experiments included going exhaustively over all ascending partitions $AP_{n,k}$ for $n=1,\ldots,200$ and all legal values of $k$ where $p_1 \geq 2$. Analysis of the case $p_1=1$ is completely understood and is discussed elsewhere (see, for example, \cite{kotlar16}). For each SSC partition $\PP$, we ran the algorithms described in Section~\ref{sec:constructing} to search for a balanced  partition implementing the partition $\PP$, and in the case that such an assignment was not found, we checked if it is non-equitable by one of the two criteria described in Section~\ref{sec:criteria}.
This enabled us to exhaustively classify all $\PP \in AP_{n,k}$ for $n \leq 200$ as one of the following:
\begin{enumerate}
    \item Non-SSC (and thus, non-equitable)
    \item Equitable, witnessed by some implementing assignment
    \item Non-equitable SSC, where criteria 1 or 2 or both hold
\end{enumerate}
Although smallest in size, the third category is the most intriguing class, as it contains the counterexamples refuting the conjecture made in \cite{miller2003} and \cite{kotlar16}, that SSC implies equitability. All algorithms where implemented in C++ and ran on a Linux workstation\footnote{Intel (R) Xeon (R) Platinum 8168 CPU @ 2.70GHz, 96 cores, 1TB total ram. The C++ program is single threaded, so probably only a small portion of available resources where utilized.}.  The total run time of the experiment was 59 hours for processing all partitions for $n \leq 200$.
In Figure~\ref{fig:number_of_partitions} one can see the number of partitions in each of the above three categories.

\begin{figure}[H]
  \centering
  \subfloat{{\includegraphics[width=0.8\textwidth,clip,trim=0mm 35mm 0mm 35mm]{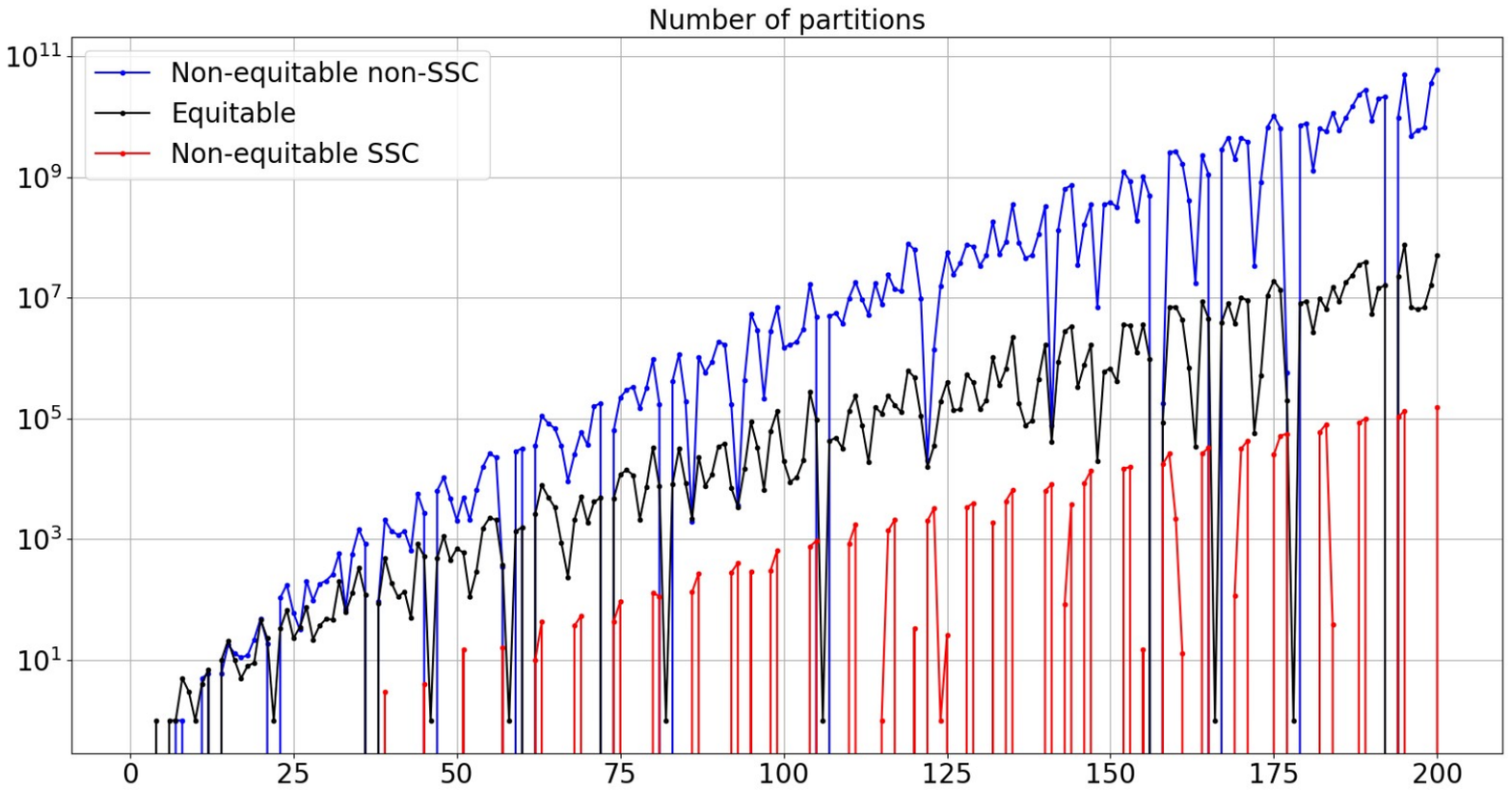}}} \\
  \subfloat{{\includegraphics[width=0.8\textwidth,clip,trim=0mm 35mm 0mm 35mm]{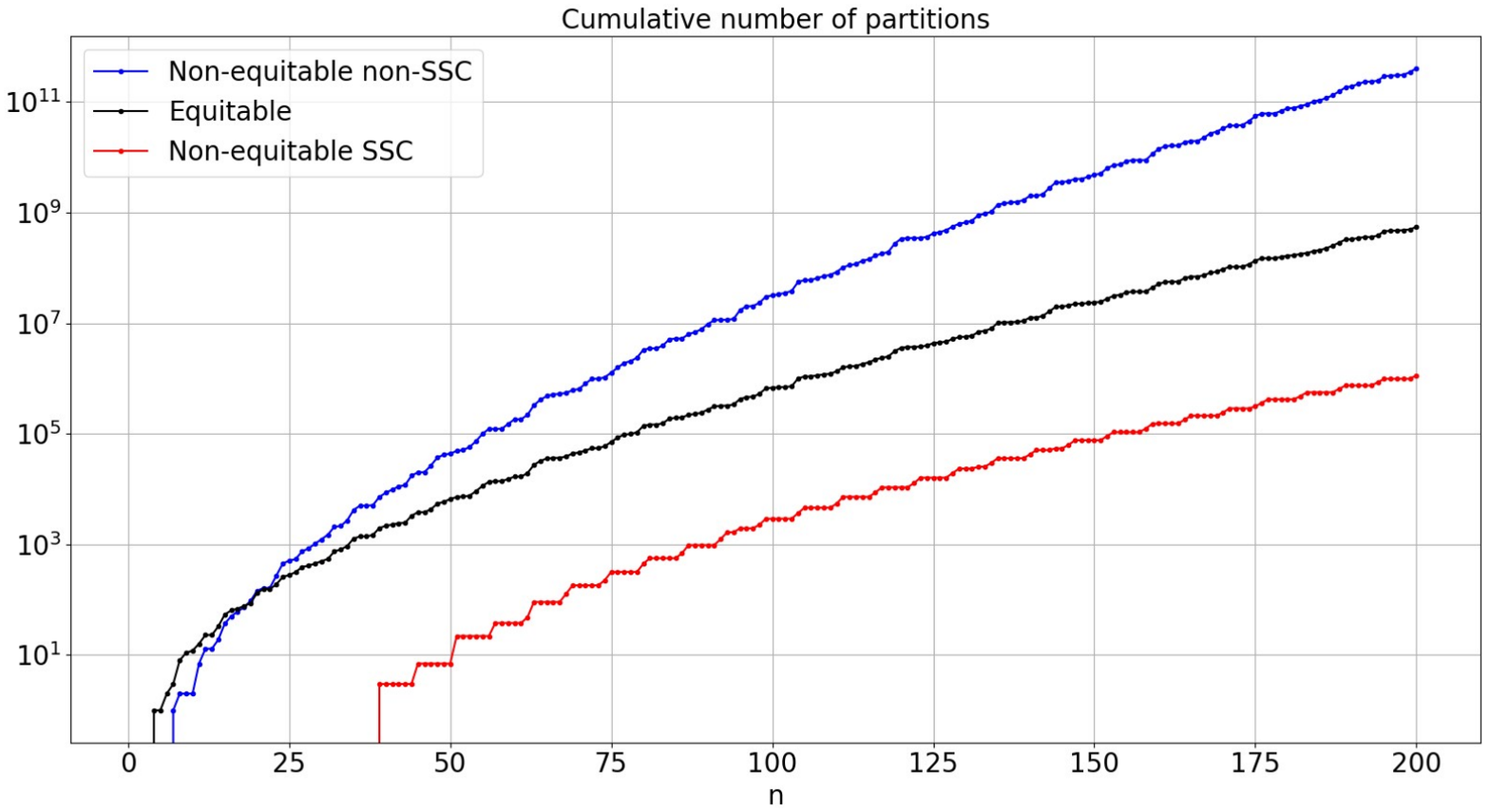}}}
  \caption{Number of ascending partitions by category as a function of $n$: non-cumulative (top), cumulative (bottom).}
  \label{fig:number_of_partitions}
\end{figure}

The 1,143,771 non-equitable SSC ascending partitions that exists for $n \leq 200$ where detected by criteria 1 or 2 or both as described in Table~\ref{tab:criteria}.
\begin{center}
\renewcommand{\arraystretch}{1.2}
\begin{table}[h]
\begin{tabular}{|c|c|}
\hline
Detected by       & Number of partitions \\
\hline
Criterion 1 only  & 1,127,545            \\
\hline
Criterion 2 only  & 12,068               \\
\hline
Criteria 1 and 2  & 4,158                \\
\hline
\end{tabular}
\caption{Non-equitable SSC partitions detected by the two criteria presented in Section~\ref{sec:criteria}}
\label{tab:criteria}
\end{table}   
\end{center}

Another way to consider the same 1,143,771 non-equitable SSC partitions for $n \leq 200$ is to classify them by distinct values of $n,k$.
These 86 $(n,k)$ pairs can be further broken down into families, characterized by the relation between $n$ and $k$. 
In each family of pairs $(n,k)$, the $n$ values, as well as the $k$ values, form an arithmetic progression. 
This suggests that the families are infinite, a fact that will be proved in Section~\ref{sec:special_families} for the two largest families.
Following are the six families noticed among the partitions in this study, and additional 14 remaining pairs listed at the end as ``other'':
\begin{enumerate}
    \setlength{\itemsep}{2pt}
    \item Total of 27 $(n,k)$ pairs for 528,201 non-equitable SSC ascending partitions with $n=3k$:
    
        \noindent{\small\begin{tabular}{ccccccccc}
             (39, 13), &  (45, 15), &  (51, 17), &  (57, 19), &  (63, 21), &  (69, 23), &  (75, 25), &  (81, 27), &  (87, 29),\\ 
             (93, 31), &  (99, 33), & (105, 35), & (111, 37), & (117, 39), & (123, 41), & (129, 43), & (135, 45), & (141, 47),\\
            (147, 49), & (153, 51), & (159, 53), & (165, 55), & (171, 57), & (177, 59), & (183, 61), & (189, 63), & (195, 65).
        \end{tabular}}
    
    \item Total of 24 $(n,k)$ pairs for 559,647 non-equitable SSC ascending partitions with $n=3k-1$:
    
        \noindent{\small\begin{tabular}{cccccccc}
            (62, 21),  & (68, 23),  & (74, 25),  & (80, 27),  & (86, 29),  & (92, 31),  & (98, 33),  & (104, 35),\\
            (110, 37), & (116, 39), & (122, 41), & (128, 43), & (134, 45), & (140, 47), & (146, 49), & (152, 51),\\
            (158, 53), & (164, 55), & (170, 57), & (176, 59), & (182, 61), & (188, 63), & (194, 65), & (200, 67).
        \end{tabular}}
        
    \item Total of 6 $(n,k)$ pairs for 16,983 non-equitable SSC ascending partitions with $n=8k/3$:
    
        \noindent{\small\begin{tabular}{cccccc}
            (80, 30), & (104, 39), & (128, 48), & (152, 57), & (176, 66), & (200, 75).
        \end{tabular}}
        
    \item Total of 5 $(n,k)$ pairs for 4,578 non-equitable SSC ascending partitions with $n=8k/3-1$:
    
        \noindent{\small\begin{tabular}{ccccc}
            (87, 33), & (111, 42), & (135, 51), & (159, 60), & (183, 69).
        \end{tabular}}
        
    \item Total of 6 $(n,k)$ pairs for 173 non-equitable SSC ascending partitions with $n=5k/2$:
    
        \noindent{\small\begin{tabular}{cccccc}
            (95, 38), & (115, 46), & (135, 54), & (155, 62), & (175, 70), & (195, 78).
        \end{tabular}}
        
    \item Total of 4 $(n,k)$ pairs for 61 non-equitable SSC ascending partitions with $n=5k/2-1$:
    
        \noindent{\small\begin{tabular}{cccc}
            (124, 50), & (144, 58), & (164, 66), & (184, 74).
        \end{tabular}}
        
    \item Total of 14 other $(n,k)$ pairs for 34,128 non-equitable SSC ascending partitions:
    
        \noindent{\small\begin{tabular}{ccccccc}
            (95, 30),  & (120, 44), & (125, 45), & (132, 42), & (143, 52), & (144, 45), & (144, 60),\\
            (153, 63), & (160, 56), & (161, 63), & (169, 65), & (175, 55), & (175, 56), & (195, 70).
        \end{tabular}}
\end{enumerate}

\vskip 2pt
The possible $(n,k)$ pairs for non-equitable SSC ascending partitions up to $n=200$ are also visualized in Figure~\ref{fig:nk_scatter} as a scatter plot in the $n-k$ plane, 
where each of the above six families correspond to a straight line. It is interesting to note that the minimal and maximal $n/k$ ratios are:
\begin{itemize}
    \item 3 partitions with $n/k = 144/60 = 2.4$:  $[2^{57}, 6^2, 18], [2^{56}, 4^2, 6, 18], [2^{55}, 4^4, 18]$.
    \item 3762 partitions with $n/k = 144/45 = 3.2$: $[2^{28}, 3^8, 4^3, 5^2, 6, 7, 8, 21], \ldots, [2^{25}, 3^{14}, 8^2, 9^4]$.
\end{itemize}

\begin{figure}[H]
  \includegraphics[width=0.8\textwidth,clip,trim=0mm 35mm 0mm 35mm]{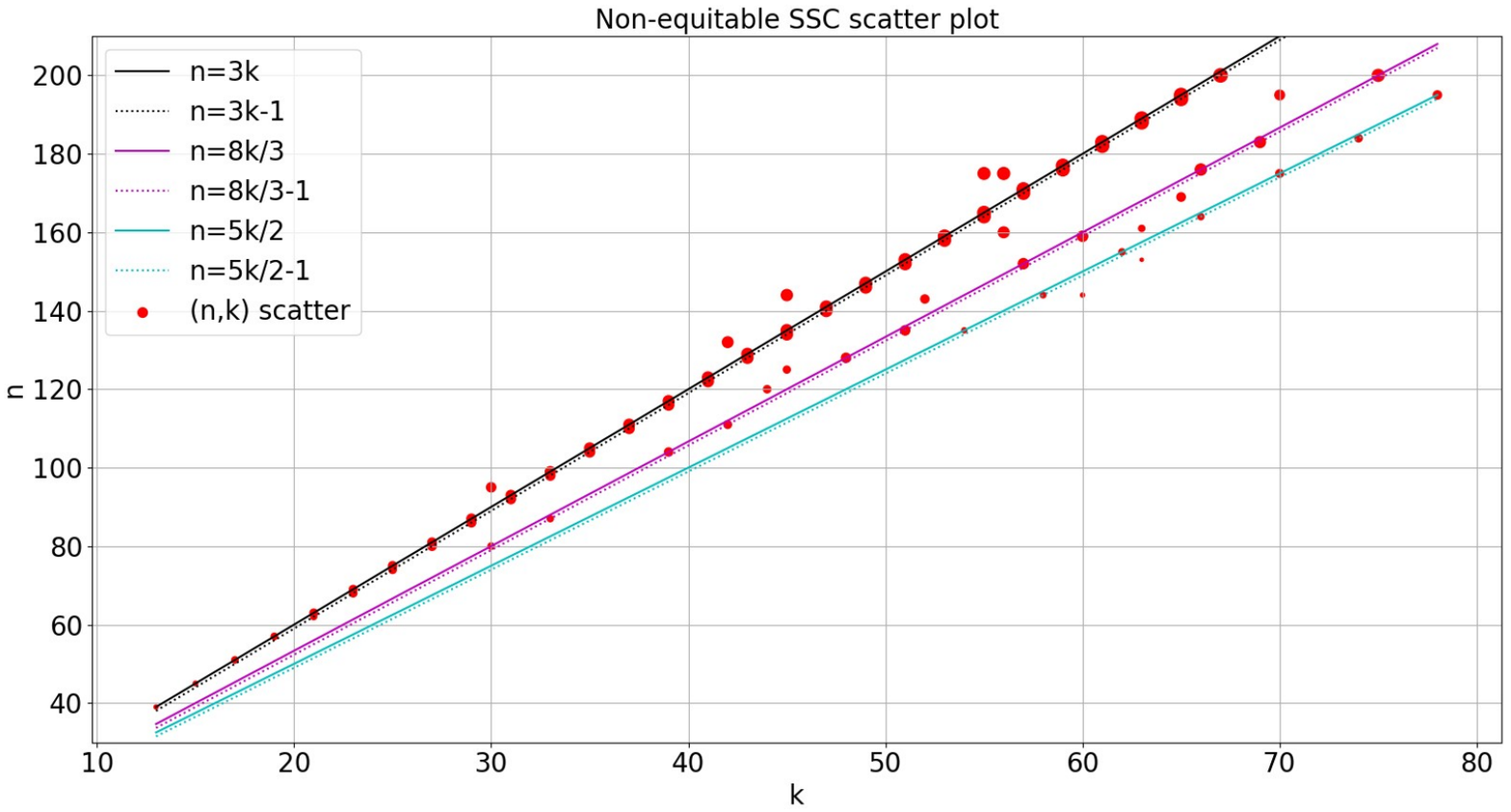}
  \caption{Scatter plot of n vs. k for all non-equitable SSC ascending partitions. The red dot size is proportional to the logarithm of the number of such partitions. The 6 lines depict the families described above.}
  \label{fig:nk_scatter}
\end{figure}

Concluding this section, we present some results on the performance of the four initialization methods described in Section~\ref{sec:constructing}.
We consider the running time of the entire solve algorithm as well as its success rate.
One should realize that the run time of solve is a bottleneck in trying to increase the value of $n$ beyond 200 and possibly uncover interesting non-equitable SSC partitions.
In the experiment, each initialization strategy was evaluated by running solve on each of the 23,838,045 equitable partitions for $n \leq 150$. The subsequent transposition algorithm, as described in Algorithm~\ref{alg:transpositions}, remained fixed throughout the experiment, with the paramaters $\textrm{MaxStepsBetweenPerturbations}=1000, \textrm{MaxPerturbations}=100$.
The results are summarized in Table~\ref{tab:init_performance}.

\vskip 5pt
\noindent
\begin{center}
\renewcommand{\arraystretch}{1.25}
\begin{table}[H]
\begin{tabular}{|l|c|c|c|c|c|c|c|}
\hline
               & Average       & Average     & Average         & \multicolumn{4}{|c|}{Probability of Failure ($\times 10^{-6}$)} \\
                                                                   \cline{5-8}
Initialization & time per      & d() after   & number of       & 1         & 2          & 10          & 100      \\
Method         & solve         & init        & transpositions  & attempt   & attempts   & attempts    & attempts \\
\hline
Random         & 711 $\mu$Sec  & 1,536,145   & 3,506           & 629       & 151        & 0           & 0        \\
\hline
Whole Set      & 539 $\mu$Sec  & 387,631     & 2,606           & 466       & 150        & 8           & 0        \\
\hline
Round Robin    & 243 $\mu$Sec  & 60,919      & 935             & 156       & 16         & 0.08        & 0.04     \\
\hline
Adaptive       &               &             &                 &           &            &             &          \\
Round Robin    & 71 $\mu$Sec   & 686         & 98              & 2.4       & 1.0        & 0.08        & 0.04     \\
\hline
\end{tabular}
\caption{Comparing the performance of balancing partitions for the 4 initialization methods}
\label{tab:init_performance}
\end{table}   

\end{center}

\hfill

In Table~\ref{tab:init_performance}, we ordered the 4 methods by running time from worst to best, but the order is the same for almost all metrics, so that Adaptive Round Robin (ARR) is also the best with respect to the diff metric after initialization and with respect to the total number of transpositions used for finding a balanced  partition.
The running time consists of a fixed initialization cost plus the cost for choosing and performing the transpositions that grows linearly with the number of transpositions used.
The fact that the running time of ARR is far better that all other methods, tells us that it is worthwhile to invest the extra cost of a more elaborate initialization scheme in order to reduce the transposition cost.
Also, when using only a single attempt to solve each partition, the ARR has the lowest failure probability.
This is probably due to the fact that better initialization implies fewer problems with local minima later.
However, as seen in the 100 attempts column, there is a ``flaw" with ARR. 
Namely, there is a very small portion of ``bad" partitions for which, even 100 independent attempts to solve them, do not yield a balanced partition. 
Other methods such as the Random initialization method, are less susceptible to the problem. 
This led to our final logic of starting with several ARR attempts for solving each partition, and for the very small portion of partitions that are left unsolved, try the other methods, until achieving a successful solution.
Table~\ref{tab:bad_partitions} shows the situation for two specific ``bad" partitions, where the probability of success was evaluated using $1000$ independent trials:

\vskip 5pt \noindent

\renewcommand{\arraystretch}{1.25}
\begin{table}[h]
\centering
\begin{tabular}{|l|l|c|c|}
\hline
\multicolumn{2}{|c|}{Partition}                      & $[2^{63}, 3^5, 5^2, 7, 9, {10}^2 13]$ & $[2^{49}, 3^2, 5^4, 13, 15]$ \\
\hline
\multicolumn{2}{|c|}{$n$}                            & 200                                   & 152                          \\
\hline
\multicolumn{2}{|c|}{$k$}                            & 75                                    & 57                           \\
\hline
\multicolumn{2}{|c|}{$\snk$}                         & 268                                   & 204                          \\
\hline
\multirow{4}{*}{$\Pr[Success]$} & Random             & 68.1\%                                & 69.8\%                       \\
\cline{2-4}
                                & WholeSet           & 82.4\%                                & 98.0\%                       \\
\cline{2-4}
                                & RoundRobin         & 14.2\%                                & 0\%                          \\
\cline{2-4}
                                & AdaptiveRoundRobin & 14.9\%                                & 0\%                          \\
\hline
\end{tabular}
\caption{Solution success rate for two ``bad'' partitions using the four initialization methods.}
\label{tab:bad_partitions}
\end{table}   




\section{Infinite families of non-equitable SSC partitions}\label{sec:special_families}

In this section, we prove that the families of non-equitable SSC partitions corresponding to pairs $(n,k)$ where $n=3k$ and $n=3k-1$ (items (1)-(2) above) are infinite. The same method of proof can be applied to show that the families listed in items (3)-(6) are infinite as well.

\begin{thm}\label{thm:counter3k}
    The partial partition $P=[2^e,3^f]$ can be completed to a non-equitable partition $\PP\in AP_{n,k}$ with nonnegative slack in the following cases:
    \begin{enumerate}
        \item [(a)] $n=6t+3$, $k=2t+1$, $e=\frac{3}{2}t$, and $f=2$, for all even $t\ge6$,
        \item [(b)] $n=6t+3$, $k=2t+1$, $e=\frac{3}{2}t-\frac{1}{2}$, and $f=3$, for all odd $t\ge7$, 
        \item [(c)] $n=6t+2$, $k=2t+1$, $e=\frac{3}{2}t$, and $f=3$, for all even $t\ge10$,
        \item [(d)] $n=6t+2$, $k=2t+1$, $e=\frac{3}{2}t+\frac{1}{2}$, and $f=2$, for all odd $t\ge11$.
    \end{enumerate}
\end{thm}

\begin{cor}\label{cor:3k}
    The classes $AP_{3k,k}$ for all odd $k\ge13$, and $AP_{3k-1,k}$ for all odd $k\ge21$ contain non-equitable SSC partitions.
\end{cor}

\begin{proof}[Proof of Corollary~\ref{cor:3k}]
    One can verify that in all four cases indeed $e,f>0$, $e+f<k$, and that both $e$ and $f$ are integral.
    The first two cases of Theorem~\ref{thm:counter3k} correspond to partitions in $AP_{3k,k}$ for all odd $k\ge13$ and the last two cases correspond to partitions in $AP_{3k-1,k}$ for all odd $k\ge21$.
\end{proof}

Before proving Theorem~\ref{thm:counter3k}, we need two auxiliary lemmas that establish sufficient conditions for being able to complete a partial partition to an SSC partition. 

\begin{lem}\label{lem:partition_completion}
    Given $n,k$ such that $\snk$ is integral and an incomplete ascending partition $P = [p_1, p_2, \ldots, p_l]$ for some $l<k$ satisfying:
    (i)  $n-\sum_{i=1}^l p_i \geq (k-l)p_l$,
    and
    (ii) $\slack_j(P) \geq 0$ for all $j \leq l$, 
    then $P$ can be completed to an ascending partition $\PP = [p_1, p_2, \ldots, p_k] \in AP_{n,k}$ with $\slack(\PP) \geq 0$.
\end{lem}

The proof of Lemma~\ref{lem:partition_completion} relies on the following Lemma.

\begin{lem}\label{lem:slack_at_end_of_block}
For any $\PP \in AP_{n,k}$, if $\slack(\PP) = \min_j \slack_j(\PP) < 0$ then there is an index $j<k$ such that $p_j < p_{j+1}$ and $\slack(\PP)=\slack_j(\PP)$. 
In other words, if negative, $\slack(\PP)$ is attained at the end of a block, but not the last block.
\end{lem}

\begin{proof}
Define $\slack_0(\PP)=0$, let $P_j = \sum_{i=1}^j p_i$ and let
$$\Delta_{j+1} := \slack_{j+1}(\PP) - \slack_j(\PP) = \sum_{i=P_j+1}^{P_{j+1}} (n-i+1) - s^{n,k} \;\;\textrm{ for }0 \leq j \leq k-1.$$
If $p_j = p_{j+1}$, then $\Delta_{j+1} < \Delta_j$, since in both $\Delta_j$ and $\Delta_{j+1}$ we sum up $p_j$ numbers, but in $\Delta_{j+1}$ the numbers are smaller.
Therefore, $\slack_{j+1}(\PP) - \slack_j(\PP) < \slack_j(\PP) - \slack_{j-1}(\PP)$, implying that $\slack_j(\PP) > (\slack_{j-1}(\PP)+\slack_{j+1}(\PP))/2$, i.e. $\slack_j(\PP)$ cannot be a strict minimum.
\end{proof}

Lemma~\ref{lem:slack_at_end_of_block} implies that for any ascending partition $\PP=[p_1^{m_1}, \ldots, p_l^{m_l}] \in AP_{n,k}$, in order to verify that $\slack(\PP)\ge0$
it suffices to check the slack condition only at the end of blocks, namely $\slack(\PP) = \min_{j \in \{m1,m1+m2,\ldots,m1+\ldots+m_{l-1}\}}\slack_j(\PP)$.

\begin{proof}[proof of Lemma~\ref{lem:partition_completion}]
    By Condition (\textit{i}) of the lemma there exist $t$ and $r$ such that $n-\sum_{i=1}^lp_i = t(k-l)+r$, $t\ge p_l$, and $0\le r<k-l$. If $r=0$, we can complete $P$ by adding $k-l$ parts, each of size $t$, and we are done, by Lemma~\ref{lem:slack_at_end_of_block}.
    So we assume $r>0$. We define the completion of the partition $P$ as:
    \begin{eqnarray*}
    p_j = \left\{ \begin{tabular}{cc}
         $t$   & for $j=l+1,\ldots,k-r$\\
         $t+1$ & for $j=k-r+1,\ldots,k$
    \end{tabular}\right.
    \end{eqnarray*}
    We claim that $\PP=[p_1,\ldots,p_k]\in AP_{n,k}$ and $\slack(\PP) \geq 0$.

    The first statement holds, since $t\ge p_l$. 
    We show that the second statement holds by contradiction. Indeed, assume that $\slack(\PP) < 0$. 
    By (ii) and Lemma~\ref{lem:slack_at_end_of_block} the slack condition must be violated at the end of the $t$-block, namely, $\slack_{k-r}(P) < 0$.
    As in the proof of Lemma~\ref{lem:slack_at_end_of_block}, we denote $\Delta_j=\slack_j(\PP)-\slack_{j-1}(\PP)$ for $j=1,\ldots,k$. 
    Since $\slack_{k-r}(P) < 0$ and $\slack_k(P) = 0$ we have:
    \begin{eqnarray}
        0 < \slack_k(P)-\slack_{k-r}(P) = \sum_{i=k-r+1}^k\Delta_i. \label{eq:arith_sum}
    \end{eqnarray}

    By Lemma~\ref{lem:slack_at_end_of_block}, $\slack_{k-r}(P) < \slack_{k-r-1}(P)$ and thus, $\Delta_{k-r}<0$.
    Moreover, $\Delta_k, \Delta_{k-1}, \ldots, \Delta_{k-r+1}$ is an ascending arithmetic progression, where $\Delta_k+\Delta_{k-r+1} = 2M$ with $M$ being its mean.
    By \eqref{eq:arith_sum} we have $M>0$ and it follows that:
    \begin{equation}\label{ineq:Delta}
        \Delta_{k-r+1}=2M-\Delta_k>-\Delta_k
    \end{equation}

    Note that for all $j$, $\Delta_j$ is the sum of $p_j$ consecutive numbers minus $\snk$. 
    So, there is a number $x$ such that $\Delta_{k-r+1}=\sum_{i=x}^{x+t}i-\snk$ and $\Delta_{k-r}=\sum_{i=x+t+1}^{x+2t}i-\snk$. Since $\Delta_{k-r}<0$ we have $\snk>\sum_{i=x+t+1}^{x+2t}i$. Also, note that $\Delta_k=\sum_{i=1}^{t+1}i-\snk$. Applying these identities and inequalities to \eqref{ineq:Delta} yields:
    \begin{equation*}
        \sum_{i=x}^{x+t}i-\snk > \snk-\sum_{i=1}^{t+1}i,
    \end{equation*}
    or
    \begin{equation*}
        \sum_{i=x}^{x+t}i>2\snk-\sum_{i=1}^{t+1}i>2\sum_{i=x+t+1}^{x+2t}i-\sum_{i=1}^{t+1}i.        
    \end{equation*}
    Applying the sum formula we obtain:
    \begin{equation*}
    \begin{split}
        (t+1)(x+\frac{t}{2})&>2t(x+\frac{3t+1}{2})-\frac{(t+1)(t+2)}{2}\\
    \Rightarrow\quad\quad    x+\frac{t}{2}&>t(2x+3t+1)-\frac{t^2}{2}-\frac{3}{2}t-1-t(x+\frac{t}{2})\\ 
    \Rightarrow\quad x(1-t)&>2t^2-t-1 = (t-1)(1+2t)
    \end{split}
    \end{equation*}

    This is a contradiction since the lhs is at most zero while the rhs is at least zero.
\end{proof}

\begin{proof}[Proof of Theorem~\ref{thm:counter3k}]
    For each of the four cases, we show that there exists a non-equitable SSC partition $\PP=[2^e,3^f,\ldots]$. 
    For this, it is sufficient to verify that the partial partition $P=[2^e,3^f]$
    satisfies the conditions of Lemma~\ref{lem:partition_completion} as well as the conditions of Criterion 1 (Theorem~\ref{thm:criterion1} with $p=3$).  
    Of the two conditions of Lemma~\ref{lem:partition_completion}, 
    Condition (i) $n-2e-3f\ge 3(k-e-f)$, trivially holds in all four cases, and by Lemma~\ref{lem:slack_at_end_of_block}, it is sufficient to check condition (ii), $slack_j(\PP) \geq 0$ for $j \leq e+f$,  only at $j=e$ and $j=e+f$.
    Therefore, we are left with four conditions that must be met, where the first two are slack conditions, and the last two are the conditions of Criterion 1 (Theorem~\ref{thm:criterion1}).
    \begin{eqnarray}
            \slack_e(\PP) &\geq& 0,        \label{eq:counter3k_cond1} \\
            \slack_{e+f}(\PP) &\geq& 0,    \label{eq:counter3k_cond2} \\            
            f &>& h,                       \label{eq:counter3k_cond3} \\
            (c-1)+(c-2)+(c-3) &<&\snk,     \label{eq:counter3k_cond4} 
    \end{eqnarray}
    where, recalling from Notation~\ref{notn:h}, $c=\snk-n$ is the smallest element in the set $C$ of numbers in $[n]$ that may be in a 2-tuple in a balanced partition, and $h=|C|-2e$ is the number of elements of $C$ that are left out after populating the 2-tuples.
    Table~\ref{tab:cases} summarizes the data in the theorem, along with the calculated values of $\snk$, $h$, $c$, $3c-6$, $e+f$ and $2e+3f$.

    \renewcommand{\arraystretch}{1.5}
    \begin{table}[h]
        \centering
        \begin{tabular}{|c|c|c|c|c|c|c|c|c|c|c|c|}
         \hline
         case   & $n$    & $k$    & $\snk$ & $c$    & $3c-6$ & $e$                        & $f$ & $h$ &$e+f$                       & $2e+3f$ & range of $t$\\
         \hline
          (a)   & \multirow{2}{*}{$6t+3$} & \multirow{2}{*}{$2t+1$} & \multirow{2}{*}{$9t+6$} & \multirow{2}{*}{$3t+3$} & \multirow{2}{*}{$9t+3$} & 
          $\frac{3}{2}t$             & 2   & 1   &$\frac{3}{2}t+2$            & $3t+6$ & $t \geq 6$ even\\
          \cline{1-1}\cline{7-12}
          (b)   &  & & & & & $\frac{3}{2}t-\frac{1}{2}$ & 3  & 2    & $\frac{3}{2}t+\frac{5}{2}$ & $3t+8$ & $t \geq 7$ odd\\
          \hline
          (c)   & \multirow{2}{*}{$6t+2$} & \multirow{2}{*}{$2t+1$} & \multirow{2}{*}{$9t+3$} & \multirow{2}{*}{$3t+1$} & \multirow{2}{*}{$9t-3$} & 
          $\frac{3}{2}t$             & 3   & 2   & $\frac{3}{2}t + 3$         & $3t+9$ & $t \geq 10$ even\\
          \cline{1-1}\cline{7-12}
          (d)   &  & & & & & $\frac{3}{2}t+\frac{1}{2}$ & 2   & 1   & $\frac{3}{2}t+\frac{5}{2}$ & $3t+7$ & $t \geq 11$ odd\\
          \hline
        \end{tabular}
        \caption{The four cases of Theorem~\ref{thm:counter3k}}
        \label{tab:cases}
    \end{table}
    \renewcommand{\arraystretch}{1}
    
    Leaving \eqref{eq:counter3k_cond2} for later, let us first consider the three other inequalities in order:
    \begin{itemize}
    \item 
        Condition \eqref{eq:counter3k_cond1} can be rewritten as:
        $$\slack_e(\PP) = 2e(n-\frac{2e-1}{2}) - e \snk = e(2n-2e+1-\snk) = e h \geq 0,$$
        by \eqref{eq:h}, which is true in all four cases, by Table~\ref{tab:cases}.
    \item 
        Conditions \eqref{eq:counter3k_cond3} and \eqref{eq:counter3k_cond4} hold in all four cases, as indicated by the data in Table~\ref{tab:cases}.
    \end{itemize}
    \noindent
    
    Therefore, the problem reduces to showing that Condition \eqref{eq:counter3k_cond2} holds, 
    namely:
    \begin{equation*}
        \slack_{e+f}(\PP) = \sum_{i=n-(2e+3f)+1}^n i -(e+f)\snk\ge 0
    \end{equation*}
    or
    \begin{equation}\label{eq:slack:e+f:1}
        (2e+3f)(n+\frac{1-(2e+3f)}{2})-(e+f)\snk\ge 0
    \end{equation}

    Substituting the expressions for $n,\snk,e+f$ and $2e+3f$ from Table~\ref{tab:cases} in \eqref{eq:slack:e+f:1} in each of the four cases and solving for $t$ yields:
    
    \noindent {\bf Case (a):}    
    
    \begin{equation*}
    \begin{split}
        (3t+6)(6t+3+\frac{1-3t-6}{2})-(\frac{3}{2}t+2)(9t+6)&\ge 0\\
        (3t+6)(\frac{9}{2}t+\frac{1}{2})-(\frac{3}{2}t+2)(9t+6)&\ge 0\\
        \frac{3}{2}t-9&\ge 0,
    \end{split}
    \end{equation*}
    
    which holds whenever $t\ge 6$.
    
    \bigskip
    
    \noindent {\bf Case (b):} 
    
    \begin{equation*}
        (3t+8)(6t+3+\frac{1-3t-8}{2})-(\frac{3}{2}t+\frac{5}{2})(9t+6)\ge 0,   
    \end{equation*}
    which holds whenever $t\ge 6.33$. Since $t$ is odd we assume $t\ge7$. 

    \hfill

    \noindent {\bf Case (c):}
    Solving for $t$ we obtain $t\ge9$, and we can assume that $t\ge 10$, since $t$ is even.

    \noindent {\bf Case (d):}
    Solving for $t$ we obtain $t\ge9.67$, and we can assume that $t\ge 11$, since $t$ is odd. 
    
    This completes the proof.
\end{proof}

\begin{rem}
    The values of $e$ and $f$ in Theorem~\ref{thm:counter3k} were chosen so that $e$ is maximal for $h>0$, and $f$ is minimal so that $f>h$. This was sufficient for showing that each one of the families $AP_{n,k}$ mentioned in Theorem~\ref{thm:counter3k} contains non-equitable SSC partitions. As the results of the computation in Section~\ref{sec:implementation} indicate, if we allow smaller values of $e$ (which results in larger $h$ and thus, larger $f$) and $f$ larger than $h+1$, while keeping the other constraints intact, we may obtain more such partitions.
\end{rem}

\section{Conclusion}\label{sec:conclusion}
The computerized search work described in Section~\ref{sec:constructing} was initiated aiming to verify the conjecture that every SSC ascending partition is equitable and possibly suggest a method for proving it. However, the search uncovered many counter-examples, suggesting new possible lines of research.

One line of research is trying to enumerate families of non-equitable SSC partitions. We have seen that the classes $AP_{3k,k}$ and $AP_{3k-1,k}$  contain infinitely many non-equitable SSC members. This is also true for the classes $AP_{\frac{8}{3}k,k}$, $AP_{\frac{8}{3}k-1,k}$, $AP_{\frac{5}{2}k,k}$, and $AP_{\frac{5}{2}k-1,k}$. So we could ask whether there is a rule about the relation between $n$ and $k$ that allows the class $AP_{n,k}$ (when non-empty) to contain non-equitable SSC partitions. We can restrict to the case of a linear relation and ask:

\begin{ques}\label{ques:alpha}
    For which values of $\alpha$ and $\beta$ does $AP_{\alpha k+\beta,k}$ contain non-equitable SSC partitions for infinitely many values of $k$?
\end{ques}

We have seen that it is easy to create constraints on equitability when the partition contains pairs, but can we do the same when the smallest part is larger than 2? in other words,

\begin{ques}
    Are there non-equitable SSC ascending partitions with $p_1>2$?
\end{ques}

A positive answer to this question requires the uncovering of new families of non-equitable SSC partitions, possibly by expanding the search to much larger values of $n$.

Now, we go back to the original question posed in Problem~\ref{prob1}. 
The fact that the necessary condition in \eqref{eq:nec_cond} is not always sufficient does not mean that there are no other sufficient conditions for equitability. 
However, the different families of non-equitable SSC ascending partitions found so far and the two criteria for non-equitability introduced in Section~\ref{sec:criteria} suggest that there might be many more families (as implied by Question~\ref{ques:alpha}) and criteria that may be hard to enumerate. 
This gives rise to the following questions:

\begin{ques}
    Is the decision problem of whether a given partition is equitable NP-complete?
\end{ques}
or alternatively,
\begin{ques}
    Is there an efficient (polynomial) procedure for deciding whether a given partition is equitable?
\end{ques}

The question is basically about the complexity and richness of the space of non-equitable SSC partitions. The problem has some resemblance to other NP-complete problems involving sums of subsets, such as  \href{https://en.wikipedia.org/wiki/Subset_sum_problem}{the subset sum problem}, \href{https://en.wikipedia.org/wiki/Partition_problem}{the partition problem} and \href{https://en.wikipedia.org/wiki/3-partition_problem}{the 3-partition problem}. 

Another possible direction of research is to identify special cases in which the conjecture holds. 
One approach may be to fix $k$. We know that the conjecture is true for $k\le 4$ \cite{miller2003, beena2009, kotlar16}, but does it hold for any fixed $k$ (and sufficiently large $n$)?
Another approach is to limit the number of distinct values in $\PP$. For example, Cichacz and G{\H{o}}rlich \cite{cichacz2018constant} solved the case where there are two distinct values, each appearing the same number of times.
We suggest a new case, which we believe is equitable as well: 

\begin{ques}[The linear case]
    Given the rational numbers $0 < a_1 \leq \cdots \leq a_k$ such that the following conditions hold:
    \begin{enumerate}
    \item $\sum_{i=1}^ka_i=1$,
    \item $\PP=\{a_1n,a_2n,\ldots,a_kn\}$ is in $AP_{n,k}$,
    \item
    \begin{eqnarray*}
        \left( \sum_{i=1}^j a_i \right) \left(1-\frac 1 2 \sum_{i=1}^j a_i\right) > \frac j {2k}   && \textrm{for all }j=1,2,\ldots,k-1,
    \end{eqnarray*}
    \end{enumerate}
    prove that $\PP$ is equitable for a sufficiently large $n$, and moreover, that $\PP$ can always be solved by the Adaptive Round Robin algorithm~\ref{alg:adaptiveRoundRobin} followed by a constant number of improving transpositions.
\end{ques}

Note that condition (3) is obtained by strengthening the slack condition:
\begin{eqnarray*}
    \slack_j(\PP) &=& \sum_{i=1}^{(a_1+\cdots+a_j)n} (n-i+1) - j \frac{n(n+1)}{2k} \\
                  &=& \left[ \left( \sum_{i=1}^j a_i \right) \left(1-\frac 1 2 \sum_{i=1}^j a_i\right) - \frac j {2k} \right] n^2 \pm o(n^2) \geq 0.
\end{eqnarray*}


\bibliographystyle{abbrv}

\end{document}